\documentclass[journal]{journal}
\let\labelindent\relax
\usepackage{amsmath}
\usepackage{amssymb}
\usepackage{amsthm}
\usepackage{authblk}

\usepackage{graphicx}
\usepackage[style=numeric,backend=bibtex8,style=ieee,sorting=none]{biblatex}
\usepackage{caption}
\usepackage{algorithm}
\usepackage{algorithmic}

\usepackage{newfloat}
\usepackage{listings}
\DeclareCaptionStyle{ruled}{labelfont=normalfont,labelsep=colon,strut=off} 
\floatstyle{ruled}
\newfloat{listing}{tb}{lst}{}
\floatname{listing}{Listing}

\usepackage[ruled,algo2e]{algorithm2e}
\usepackage{amsmath}
\usepackage{amssymb}
\usepackage{amsthm}
\usepackage{breqn}
\usepackage{enumerate}
\usepackage{enumitem}
\usepackage[T1]{fontenc}
\usepackage[hang,flushmargin]{footmisc}
\usepackage{multicol}
\usepackage{rotating}
\usepackage{tabularx}
\usepackage{tikz}
\usepackage{thmtools}
\usepackage{thm-restate}
\usepackage{todonotes}
\usepackage{xcolor}
\usepackage{xfrac}

\usepackage[capitalize]{cleveref} %

\newtheorem{Assumption}{Assumption}
\newtheorem{Corollary}{Corollary}

\newtheorem{Lemma}{Lemma}
\newtheorem{Proposition}{Proposition}
\newtheorem{Remark}{Remark}

\Crefname{Assumption}{Assumption}{Assumptions}
\theoremstyle{definition}
\theoremstyle{remark}

\definecolor{OliveGreen}{rgb}{0,.35,0}
\definecolor{RoyalBlue}{cmyk}{1, 0.50, 0, 0.25}

\SetKwInOut{Input}{Input}
\SetKwInOut{Output}{Output}
\author[1]{ Bastien Batardière}
\author[1]{Joon Kwon}

\affil[1]{Université Paris-Saclay, AgroParisTech, INRAE, UMR MIA Paris-Saclay\newline
bastien.batardiere@inrae.fr, joon.kwon@inrae.fr}

\DeclareMathOperator*{\argmin}{arg\,min}

\newcommand{\anchor}{y^{(t-1)}}
\newcommand{\barre}{\overline}
\newcommand{\bin}[1]{}

\newcommand{\diag}{\operatorname{diag}}
\newcommand{\dln}{\left\|}
\newcommand{\drn}{\right\| }
\newcommand{\esp}[1]{\mathbb E \left[ #1 \right]}
\newcommand{\espt}[1]{\mathbb E_t \left[ #1 \right]}

\newcommand{\esptnd}[1]{\mathbb E_t \left[ \dln #1 \drn ^2  \right]}
\newcommand{\gt}{g^{(t)}}
\newcommand{\gtop}{g^{(t)\top}}

\newcommand{\mE}{\mathbb E }

\newcommand{\mR}{\mathbb R }
\newcommand{\meann}{\frac 1 n \sumunn}
\newcommand{\meanT}{\frac 1 T \sumunT}
\newcommand{\mt}{m^{(t)}}
\newcommand{\mtm}{m^{(t-1)}}
\newcommand{\PI}[1]{\Pi_{X, A_t}\left(#1\right)}

\newcommand{\ps}[2]{\left\langle #1, #2 \right\rangle}
\newcommand{\sumunn}{\sum_{i= 1}^n}
\newcommand{\sumunT}{\sum_{t= 1}^T}
\newcommand{\tildexti}{\tilde x^{(t)}_i}

\newcommand{\tildextmi}{\tilde x^{(t-1)}_i}
\newcommand{\tildextmri}{\tilde x^{(t-1)}_{i(t)}}
\newcommand{\Tr}{\operatorname{Tr}}
\newcommand{\xt}{x^{(t)}}

\newcommand{\xtp}{x^{(t+1)}}
\newcommand{\xstar}{x^{\star}}
\newcommand{\wt}{\tilde{x}^{(t)}}
\newcommand{\wtm}{\tilde{x}^{(t-1)}}

\renewcommand{\leq}{\leqslant}
\renewcommand{\geq}{\geqslant}

\ifCLASSINFOpdf
\else
\fi
\begin{document}
\title{Finite-Sum Optimization: Adaptivity to Smoothness and
  Loopless Variance Reduction}

\maketitle
\thispagestyle{empty}

\begin{abstract}
  For finite-sum optimization, variance-reduced gradient methods (VR)
  compute at each iteration the gradient of a single function (or of a
  mini-batch), and yet achieve faster convergence than SGD thanks to a
  carefully crafted lower-variance stochastic gradient estimator that
  reuses past gradients. Another important line of research of the
  past decade in continuous optimization is the adaptive algorithms
  such as AdaGrad, that dynamically adjust the (possibly
  coordinate-wise) learning rate to past gradients and thereby adapt to
  the geometry of the objective function. Variants such as RMSprop and
  Adam demonstrate outstanding practical performance that have
  contributed to the success of deep learning. In this work, we
  present AdaLVR, which combines the AdaGrad algorithm with
  \emph{loopless} variance-reduced gradient estimators such as SAGA or L-SVRG
  that benefits from a straightforward construction and a streamlined analysis. We
  assess that AdaLVR inherits both good convergence properties from VR
  methods and the adaptive nature of AdaGrad: in the case of
  $L$-smooth convex functions we establish a gradient complexity of
  $O(n+(L+\sqrt{nL})/\varepsilon)$ without prior knowledge of $L$. Numerical
  experiments demonstrate the superiority of AdaLVR over
  state-of-the-art methods. Moreover, we empirically show that the
  RMSprop and Adam algorithm combined with variance-reduced gradients
  estimators achieve even faster convergence.
\end{abstract}

\IEEEpeerreviewmaketitle
\section{Introduction}
We consider the finite-sum optimization problem where we aim at
minimizing a function $f:\mathbb{R}^d\to \mathbb{R}$ with a finite-sum
structure: \[ f(x)=\frac{1}{n}\sum_{i=1}^nf_i(x),\quad x\in
\mathbb{R}^d, \] where each function $f_i:\mathbb{R}^d\to \mathbb{R}$
is differentiable. This formulation is ubiquitous in machine
learning (empirical risk minimization, such as least-squares or
logistic regression with e.g. linear predictors or neural networks)
and statistics (inference by maximum likelihood, variational
inference, generalized method of moments, etc.). In such applications,
it is nowadays commonplace to deal with very large dimension ($d\gg
1$) and datasets ($n\gg 1$), and gradient methods (iterative
algorithms which access the objective function by computing values and
gradients at query points) offer the best scalability.

The most basic gradient method is the Gradient Descent \parencite[GD, ][]{cauchy1847methode} which
for instance achieves a convergence rate of $1/T$ (where $T$ is the number of
iterations) in the case of a smooth objective function (meaning the
gradient is Lipschitz-continuous). However, it requires the
exact computation of a gradient of $f$ at each iteration, corresponding
to the computation of the gradients of each individual
function $f_i$ ($1\leq i\leq n$), which is costly
when $n$ is large. In such a finite-sum optimization context, the
algorithm is also called \emph{Batch} GD.
An efficient solution for improving the scalability in
$n$ is the Stochastic Gradient Descent \parencite[SGD, ][]{robbins1951stochastic} which replaces
the exact gradient computation by a noisy one: it draws one or a
small subset of indices in $\left\{ 1,\dots,n \right\}$ and computes the
average of the corresponding gradients, which is called the
stochastic gradient estimate. This much cheaper iteration
cost has a drawback: the convergence rate deteriorates into $\sqrt{T}$.
When the objective function is also strongly convex, Batch GD achieves
a linear convergence, whereas SGD guarantees a $1/T$ rate.

About a decade ago, the pioneering work of \cite{roux2012stochastic}, followed by \textcite{schmidt2017minimizing},
introduced the SAG algorithm which combined the best of both worlds by computing, at each
iteration, the gradient of a single function $f_i$ only, while
nevertheless retaining the fast $1/T$ convergence rate in the case of
smooth functions (instead of $1/\sqrt{T}$). The SAG algorithm
introduces a sophisticated gradient estimate, which reuses previously
computed gradients to implement the control variates methods~\parencite{blatt2007convergent} and reduce its variance. As the iterates converge towards
a solution, the variance of the estimates vanishes, thus recovering the
$1/T$ convergence rate from Batch GD. When the objective function
is strongly convex, SAG recovers a linear convergence rate.

Numerous alternative algorithms implementing similar ideas quickly
followed, e.g. SVRG~\parencite{johnson2013accelerating}, SAGA
\parencite{SAGA_article}, SVRG++
\parencite{allen2016improved}, L-SVRG~\parencite{hofmann2015variance},
JacSketch~\parencite{gower2021stochastic},
MISO~(\parencite{defazio2014finito,mairal2013optimization}), S2GD
(\parencite{konevcny2017semi}), L2S~\parencite{li2020convergence}, etc. Other variance-reduced algorithms computing the convex conjugates of
the functions $f_i$ (instead of gradients) were also developped: SDCA ~\parencite{shalev2013stochastic} and its extensions (e.g.\
\cite{lin2014accelerated,shalev2014accelerated}) use a coordinate
descent algorithm to solve the dual problem, when the primal problem
takes the form of a $\ell_2$-regularized empirical risk minimization.

We focus on the case where each function $f_i$ is assumed to be convex
and $L$-smooth (meaning the gradient is $L$-Lipschitz continuous,
which includes important problems such as least-squares linear
regression and logistic regression) and consider the gradient
complexity of the algorithms to attain an $\varepsilon$-approximate solution.

The basic algorithms SAG, SAGA, SVRG
and L-SVRG enjoy a gradient complexity of order
$O((n+L)/\varepsilon)$, which can be improved to
$O(n+L\sqrt{n}/\varepsilon)$~\parencite{reddi2016stochastic},
and the SVRG++ algorithm achieves $O(n\log
(1/\varepsilon)+L/\varepsilon)$~\parencite{allen2016improved}. None of those two complexities
dominate the other and are, to the best of our knowledge, the best
that does not involve Catalyst-based~\parencite{lin2015universal} or momentum-based acceleration~\parencite{allen2017katyusha}.

Combining various acceleration techniques with variance-reduced
gradient estimates is an important line of work, which managed to obtain the even more
suprising $1/T^2$ convergence rate while computing only $O(1)$
gradients by iteration~(\parencite{lan2018optimal,lan2019unified,joulani2020simpler,song2020variance,li2021anita}). In particular, the recent DAVIS
algorithm~\parencite{liu2022kill} achieves
the optimal $O(n+\sqrt{nL/\varepsilon})$ complexity, which matches the
lower bound from~\cite{woodworth2016tight}.

 In the smooth nonconvex setting, specifically designed variance reduced methods (e.g.\ SARAH~\parencite{nguyen2017sarah},
Natasha~\parencite{allen2017natasha}, SPIDER~\parencite{fang2018spider} and
STORM~\parencite{cutkosky2019momentum}) achieve better rates than the algorithms mentioned above.

An important issue with all the aforementioned algorithms is the need
to precisely tune the step-size as a function of the smoothness
coefficient \(L\) in order to achieve the convergence rates given by the
theory. If the chosen step-size is too small, the convergence may be
excruciatingly slow, and if it is too large, the iterates may very
well not converge at all. Unfortunately, prior knowledge of smoothness
coefficients can be computationally intractable when the dimension of
the problem is high (\parencite{latorre2020lipschitz}), except in some
special cases, such as least-squares linear regression or logistic
regression. Therefore, in most practical cases, tuning is performed by
running the algorithm multiple times with different values for the
step-sizes, which is tedious and resource-intensive. Therefore, the
search for algorithms that are \emph{adaptive} to an unknown smoothness
coefficient is an important challenge, and several kinds of algorithms
have been proposed.

In the context of batch optimization, line search methods try, at each
iteration, several values for the step-size until one satisfies a
given condition (e.g.\ the Armijo condition \parencite{armijo1966minimization}. This
involves computing the objective value at each corresponding candidate
iterate. Such methods provably are able to adapt to an unknown
smoothness coefficient and achieve optimal convergence rates for
convex problems (\parencite{nesterov2015universal}).

Still in batch optimization, bundle-level methods keep in memory a
subset of the past gradient information to create a model of the
objective function (\parencite{lemarechal1995new}). They
exhibit good performance in practice and they achieve optimality for smooth convex
optimization, adaptively to an unknown smoothness coefficient (\parencite{lan2015bundle}).

Other methods use the last two observed gradients to compute a simple
local estimate of the smoothness coefficient to be used in the
step-size. One popular example is the step-size proposed by \cite{barzilai1988two}: although it has good empirical performance, convergence
guarantee are only known for quadratic objectives.
A similar one is the ALS step-size (\parencite{armijo1966minimization,vrahatis2000class}), and its YK variant
\parencite{malitsky2020adaptive}. The latter adaptively achieves the same
rates as GD without prior knowledge of smoothness (and strongly
convexity) coefficients.

Although the above methods may enjoy good guarantees and performance,
an important drawback is the difficulty to extend them to the
stochastic and finite-sum optimization settings, because they rely on
exact gradient computations, although some attempts have been made---see
e.g. \cite{vaswani2019painless,dvinskikh2019adaptive} for
stochastic versions of the Armijo rule, and
\cite{tan2016barzilai,liu2019class,li2019adaptive,li2020almost,yang2021accelerating}
for the Barzilai-Borwein step-size combined with variance reduction in
finite-sum optimization.

An important family of algorithms, on which this work builds upon, is
the AdaGrad \parencite{mcmahan2010adaptive,Adagrad_article}
algorithm and its many variants, that can naturally handle finite-sum
and general stochastic first-order optimization. Coming from the
online learning literature, these algorithms compute step-sizes based on
an accumulation of all previously observed gradients, and famously
allow for coordinate-wise step-sizes. The latter have been highly
successful in deep learning optimization (with e.g. RMSprop, Adam \parencite{kingma2015adam},
etc.) where the magnitude of the partial derivatives strongly depends
on the depth of the corresponding weights in the neural network
structure. In the batch optimization context, they have demonstrated
various adaptive properties, to smoothness in particular: the AdaNGD
variant \parencite{levy2017online} was the first such algorithm to be proven adaptive to
smoothness. AdaGrad itself is established to be adaptive to smoothness
and noise in \cite{levy2018online} as well as accelerated
variants AcceleGrad \parencite{levy2018online} and UnixGrad \parencite{kavis2019unixgrad}. Many subsequent works established similar
adaptivity for other variants for the convex setting (\parencite{wu2018wngrad,cutkosky2019anytime,joulani2020simpler,ene2021adaptive,ene2022adaptive}), and for the nonconvex setting (\parencite{wu2018wngrad,ward2020adagrad,levy2021storm+,kavis2022high,faw2022power,attia2023sgd}).

\paragraph{Contributions}
\label{sec:contributions}
In the context of finite-sum optimization, we introduce AdaGrad-like
algorithms combined with SAGA and L-SVRG type gradient estimators. We
analyze the variants AdaGrad-Norm and AdaGrad-Diagonal with both
estimators in a \emph{unified fashion}, which is a key feature of this work.
Unlike \cite{liu2022adaptive}, we allow for diagonal scalings (i.e.
    per-coordinate step-sizes) which have the well-documented benefit of
    remedying
ill-conditioning (and which is demonstrated in our numerical experiments). In the case of convex functions, we
establish a $O(n+(L+\sqrt{nL})/\varepsilon)$ gradient complexity, surpassing \cite{dubois2022svrg} when $n$ is large.  Importantly,
these guarantees are adaptive to the smoothness coefficient $L$ of the
functions $f_i$, meaning that the algorithms need not be tuned with
prior knowledge of $L$. An important contribution of this work is the
    analysis of AdaGrad with \emph{loopless} gradient estimators (SAGA and L-SVRG), as
opposed to SVRG-type gradient estimators involving cumbersome inner and outer loops, resulting in a much more streamlined construction and analysis.
Numerical
experiments also consider as
heuristics the RMSprop and Adam algorithms combined with SAGA and
L-SVRG gradient estimators, and demonstrate the excellent performance of
all the proposed methods.

\paragraph{Related work}
\label{sec:related-work}

Several recent works proposed finite-sum optimization methods combining adaptivity to
smoothness and variance reduction. The closest works to ours are
\cite{dubois2022svrg} and \textcite{liu2022adaptive}. The former achieves adaptivity to smoothness by
restarting AdaGrad after each inner loop of the SVRG algorithm, and by
restarting the outer loop itself with doubling epoch lengths. This
construction achieves a \(O(n\log (L/\varepsilon)+L/\varepsilon)\) gradient
complexity in the smooth convex case (similar to SVRG++), which is
better than ours in the small \(\varepsilon\) regime and worse in the
large \(n\) regime. The latter work also combines AdaGrad with SVRG and additionally
incorporates acceleration and proximal
operators (for dealing with composite objectives) which yields the AdaVRAG
and AdaVRAE algorithms. They achieve better gradient complexities
\(O(n\log \log n+\sqrt{nL/\varepsilon})\) and \(O(n\log \log n+\sqrt{(nL\log
L)/\varepsilon})\) in the smooth convex case, with adaptivity to
smoothness; but the drawbacks are the limitation to algorithms
corresponding to AdaGrad-Norm (therefore, no diagonal scaling is
considered) and to SVRG-like variance reduction which involves inner
and outer loops. In comparison, our construction and analysis are much
more straightforward due to loopless variance reduction.

Further related works include \cite{li2022variance} which
combines AdaGrad with SVRG as a heuristic only, \cite{wang2022divergence}
which combines the Adam algorithm with SVRG,
and \cite{kavis2022adaptive} which
achieves adaptivity to smoothness by combining AdaGrad-Norm with a
SPIDER-type variance reduction, also in a nonconvex setting.

Beyond finite-sum optimization, some papers have proposed similar
approaches for general stochastic optimization:
\cite{cutkosky2018distributed} achieves adaptivity to smoothness by
combining SVRG with abstract adaptive
algorithms (including AdaGrad) in a convex setting,
whereas \cite{li2023convergence} combines Adam with a variance
reduction technique similar to STORM in a nonconvex setting.

Other kinds of adaptive step-sizes were also combined with variance
reduction: \cite{xie2022adaptive} considers BB, ALS and YK step-sizes
in conjunction with SAGA and achieves adaptivity to smoothness in the
smooth and strongly convex case; \cite{shi2021ai} proposes the
AI-SARAH algorithm that is also adaptive to smoothness in the strongly
convex case by however accessing the Hessian of individual
functions \(f_i\), which may be costly except for some simple cases;
and \cite{gower2016stochastic} which associates the celebrated
quasi-Newton BFGS algorithm with SVRG.

\section{Setting}
Let $n,d\geqslant 1$ be integers. In the following, $\dln \, \cdot \, \drn$ denotes
the canonical Euclidean norm, and for a positive semidefinite
matrix $A \in \mR^{d\times d}$, we denote the associated Mahalanobis norm as
$\dln x  \drn_A = \ps{x}{Ax}^{\sfrac 12} \;(x \in \mR^d)$. If $A$ is a
nonnegative scalar, we make an abuse of notation by denoting $\dln \cdot
\drn_A = \dln \cdot \drn_{AI_d}$.  We consider the minimization of functions with a finite-sum structure:
$$\min_{x \in X}f(x),\qquad f(x)= \meann f_i\left(x\right),$$
where $X$ is a compact convex subset of $\mathbb R^d$ with finite diameter $D \triangleq \sup_{x,y \in X}\dln x - y \drn < \infty $  and each function $f_i:\mathbb{R}^d\to \mathbb{R}, \, \left( 1 \leq i \leq n\right)$ satisfies the following assumptions.
\begin{Assumption}[Convexity and smoothness]
  \label{ass:convex-smooth}
  There exists $L>0$ such that for all $1\leqslant i\leqslant n$, $f_i$ is convex,
  differentiable, and $L$-smooth (i.e.\ $\nabla f_i$ is $L$-Lipschitz continuous).
\end{Assumption}
\begin{Assumption}[Existence of a minimum]\label{ass:minimum}
    There exists $\xstar \in X$ such that $$f\left(\xstar\right) = \min_{x\in \mathbb R^d  }f(x)= \min_{x\in X  }f(x) .$$
\end{Assumption}

\section{AdaGrad with loopless variance reduction}\label{sec:AdaLVR}
We introduce and analyze a family of algorithms that combines
AdaGrad (either the Norm or the Diagonal version) with SAGA and L-SVRG type
variance-reduced gradient estimators. For smooth convex
finite-sum optimization, we establish in Theorem~\ref{thm:AdaLVR} a
convergence rate of $O((L+\sqrt{nL})/T)$,
which corresponds to a gradient complexity of $O(n+(L+\sqrt{nL})/\varepsilon)$.
Importantly, this guarantee is achieved without prior knowledge of
the smoothness coefficient $L$.

\paragraph{AdaGrad} \label{subsec:AdaGrad} AdaGrad
\parencite{mcmahan2010adaptive,Adagrad_article} is a variant of (online)
gradient descent with a step-size that dynamically adapts to the past
observed gradients. Given a hyperparameter $\eta > 0$, an initial point
$x^{(1)} \in \mR^d$ and $\left( \gt \right)_{t \geq 1}$ a sequence in
$\mR^d$, the corresponding iterates are defined as
\begin{equation}
  \tag{AdaGrad}
\label{eq:adagrad}
x^{(t+1)} = \PI{\xt - \eta A_t^{-1} g^{(t)}}, \quad t \geq 1,
\end{equation}
where $\PI{z} \triangleq \argmin_{y \in X} \dln y-z \drn_{A_t}$ and where for $A_t$ we consider two possibilities:
\begin{itemize}
\item $A_t =  \left( \sum_{s=1}^t \dln g^{(s)}  \drn^2  \right) ^{\sfrac 1 2 } \in \mR,$\hfill (AdaGrad-Norm)
\item $A_t = \left(\sum_{s=1}^t \diag\left( g^{(s)} g^{(s)\top}\right) \right) ^{\sfrac 1 2 }\in \mR^{d\times d}$, \hfill  (AdaGrad-Diagonal)
\end{itemize}
where $\diag(A)$ denotes the matrix obtained by zeroing the
off-diagonal coefficients from $A$. In
the first case, $A_t$ is a scalar. In the second case, the algorithm
has per-coordinate step-sizes: $A_t$ is a
diagonal matrix of
size $d\times d$, the square root is to be understood component-wise and
$A_t^{-1}$ denotes the Moore--Penrose pseudo-inverse in cases where $A_t$ is
non-invertible. We also set $A_0=0$ by convention.

The following classical inequality will be used in Section~\ref{sec:comb-adagr-with} for the
analysis of AdaGrad combined with loopless variance reduction. Our
definition does not include a positive initial term $\varepsilon$ in the
definition of $A_t$, unlike most works on AdaGrad;
we therefore include the proof in the appendix for completeness.
\begin{restatable}[\cite{Adagrad_article}]{Lemma}{AdaGradLemma}\label{lem:convenientBoundAdaGrad}
The iterates defined in \eqref{eq:adagrad} satisfy
\[ \sumunT \ps{g^{(t)} }{\xt - \xstar} \leq \alpha\left(\eta + \frac {D^2} {2\eta} \right)\sqrt{\sumunT \dln \gt \drn^2}, \]
where $\alpha=1$ for AdaGrad-Norm (resp. $\alpha=\sqrt{d}$ for AdaGrad-Diagonal).
\end{restatable}

\paragraph{SAGA and L-SVRG}
\label{sec:saga-l-svrg}

The very first variance-reduced algorithm for finite-sum optimization
is SAG (\parencite{roux2012stochastic,schmidt2017minimizing}) and was
quickly followed by SVRG \parencite{johnson2013accelerating}, SAGA
\parencite{SAGA_article} and L-SVRG \parencite{hofmann2015variance}, among others.
These algorithms are all variants of SGD,
where the basic finite-sum stochastic gradient estimator (which
at each step, computes the gradient of a single function $f_i$) was replaced
by a variance-reduced one. SVRG and L-SVRG compute and update in memory
a full gradient $\nabla f(x^{(t)})=\frac{1}{n}\sum_{i=1}^n\nabla f_i(x^{(t)})$ once
in a while, whereas SAG and SAGA compute and update a single gradient
$\nabla f_{i(t)}(x^{(t)})$ at each step.

We focus in this work on the SAGA and L-SVRG type gradient estimators
because they share handy theoretical properties. As a matter of fact,
they have been analyzed in a unified way in several works
(\parencite{hofmann2015variance,gorbunov2020unified,condat2022murana}).
 They are sometimes called \emph{loopless}
by contrast to SVRG and many of its variants, which are defined with
inner and outer loops making their analysis cumbersome in comparison.
We now quickly recall their respective constructions.

Given an initial point $x^{(1)} \in \mR^d$, the SAGA algorithm sets $\tilde
x_i^{(0)} = x^{(1)} \,\,(1 \leq i \leq n)$. Then, at each iteration $t \geq 1$,
$i(t) \sim \text{Unif}\{ 1, \dots, n\}$ is sampled,
an unbiaised estimator of $\nabla f\left(\xt\right)$ is computed as
\begin{equation}
    \tag{SAGA}
\label{eq:saga}
\textstyle \nabla f_{i(t)}\left(\xt\right)  -  \nabla f_{i(t)} \left(\tilde
x^{(t-1)}_{i(t)}\right)  +  \frac 1 n\sum\limits_{i=1}^n  \nabla f_i\left(\tilde x^{(t-1)}_i\right),
\end{equation}
and $\tilde x_{i}^{(t)}$ is set to $\xt$ for $i=i(t)$ and to $\tilde x_{j}^{(t-1)}$ otherwise.

The L-SVRG algorithm first sets $\tilde{x}^{(0)}$ to $x^{(1)}$. At
each iteration $t \geq 1$, $i(t) \sim \text{Unif}\{ 1,
\dots, n\}$ is sample and an unbiaised estimator of
$\nabla f \left(\xt\right)$ is computed as
\begin{equation}
  \tag{L-SVRG}
\label{eq:LSVRG}
\nabla f_{i(t)}\left(\xt\right)  -  \nabla f_{i(t)} \left(\wtm \right)+ \nabla f\left(\wtm\right),
\end{equation}
and $\wt$ is set to $x^{(t)}$ with probability $p$ and to $\wtm$ with
probability $1-p$, where $p\in (0,1)$ is a hyperparameter (often equal
to $1/n$).

\paragraph{Combining AdaGrad with SAGA and L-SVRG}
\label{sec:comb-adagr-with}
\begin{algorithm2e*}
\caption{AdaLVR \label{alg:AdaLVR} }
\Input{$x^{(1)} \in \mR^d$ initial point, $T \geq 1$ number of
iterations, $\eta >0 $ hyperparameter, $p\in (0,1)$.}
$G_0 \gets 0$\\
$\begin{cases}
\tilde{x}^{(0)}_i\leftarrow x^{(1)},\quad (1\leq i\leq n)&\text{(SAGA)}\\
\tilde{x}^{(0)}\leftarrow x^{(1)}&\text{(L-SVRG)}
\end{cases}$\\
\For{$ t = 1 \dots T-1$}{ Sample
$i(t) \sim \text{Unif}\{1,\dots,n\}$ \\ Compute
$\displaystyle \gt = \begin{cases}\nabla f_{i(t)}\left(\xt\right) - \nabla f_{i(t)}\left(\tilde x^{(t-1)}_{i(t)}\right) + \meann \nabla f_i\left(\tilde x^{(t-1)}_i\right)&\text{(SAGA)} \\ ~~\text{
or} \\ \nabla f_{i(t)}\left(\xt\right) - \nabla f_{i(t)} \left(\wtm \right) + \nabla f\left(\wtm\right)&\text{(L-SVRG)} \end{cases}$
\\ $\displaystyle
\begin{cases}
\text{Update }\tilde{x}_i^{(t)}=x^{(t)}\mathbf{1}_{\left\{ i=i(t) \right\}}+\tilde{x}_i^{(t-1)}\mathbf{1}_{\left\{ i\neq i(t) \right\}}&\text{(SAGA)}\\\text{or}\\
\text{sample $Z^{(t)}\sim \mathcal{B}(p)$ update }\tilde{x}^{(t)}=x^{(t)}\mathbf{1}_{\left\{ Z^{(t)}=1 \right\}}+x^{(t-1)}\mathbf{1}_{\left\{ Z^{(t)}=0 \right\}}&\text{(L-SVRG)}
\end{cases}$ \\
Update $G_{t}=G_{t-1}+ \begin{cases}\dln\gt\drn^{2} & \text {
(AdaGrad-Norm) } \\\text{or}\\ \operatorname{diag}\left(\gt \gtop\right) & \text
{ (AdaGrad-Diagonal) } \end{cases} $\\
Update $\xtp = \PI{\xt - \eta G_t^{-1/2}\gt}$\
}
\Output{$\barre x _T = \meanT \xt$.}
\end{algorithm2e*}
\begin{figure*}
\begin{center}
    \includegraphics[scale=0.23]{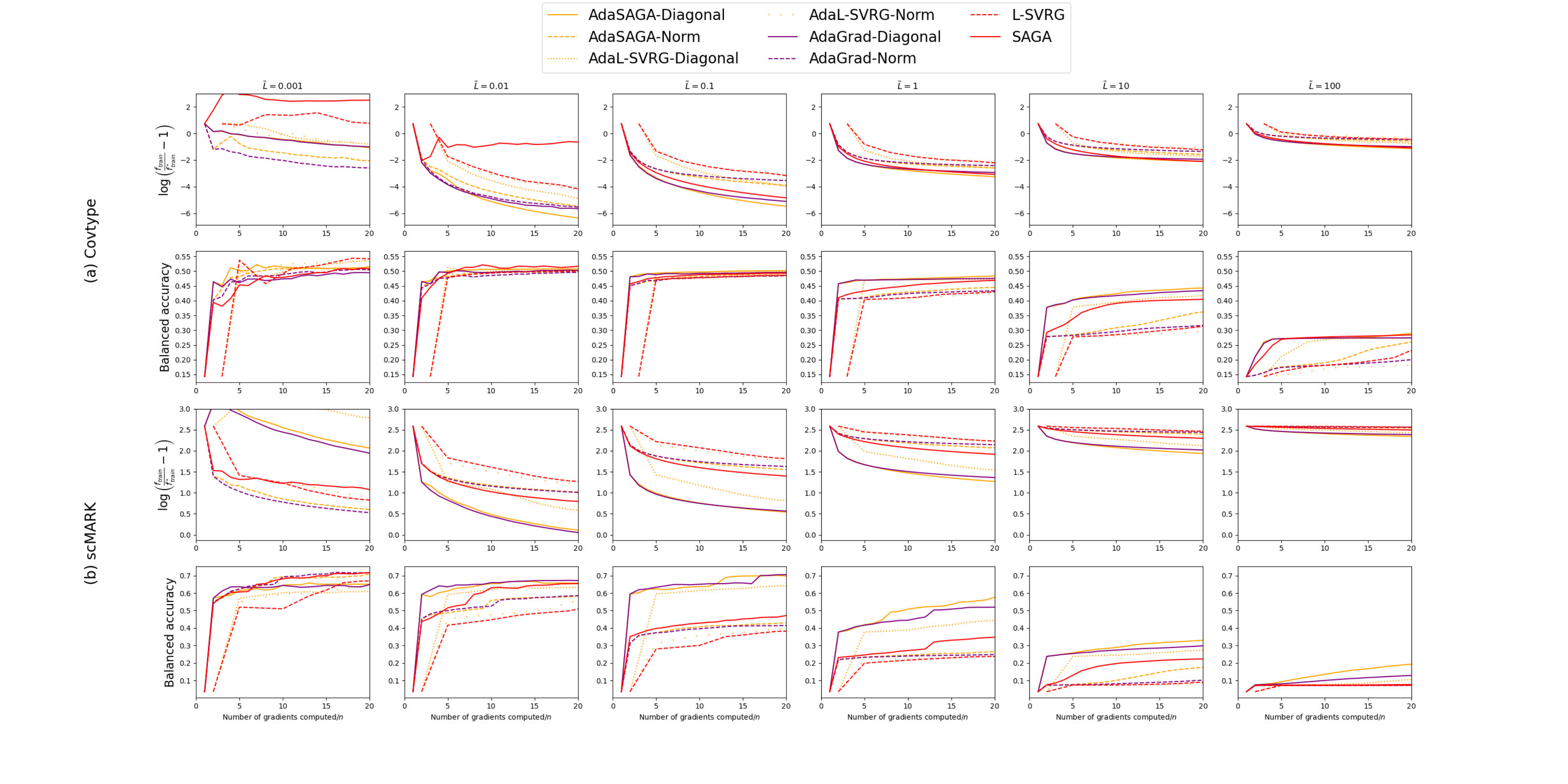}
\end{center}
\caption{Comparison of AdaLVR against SAGA, L-SVRG and AdaGrad.}
\label{fig:AdaptVRvsEach}
\end{figure*}
\begin{figure*}
\begin{center}
    \includegraphics[scale=0.23]{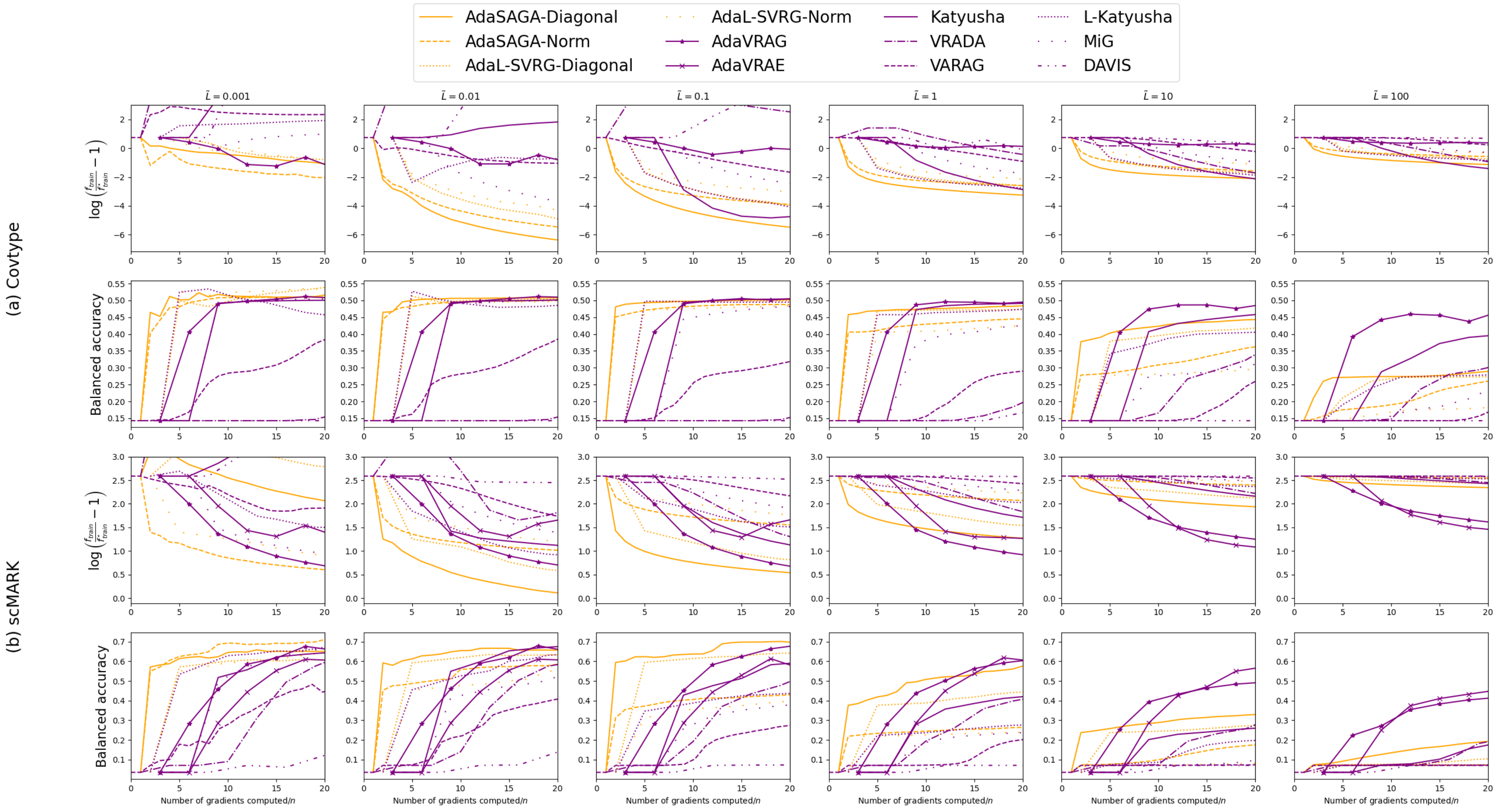}
\end{center}
\caption{Comparison of AdaLVR against accelerated algorithms with variance reduction.}
\label{fig:vsAccel}
\end{figure*}
\begin{figure*}
\begin{center}
    \includegraphics[scale=0.23]{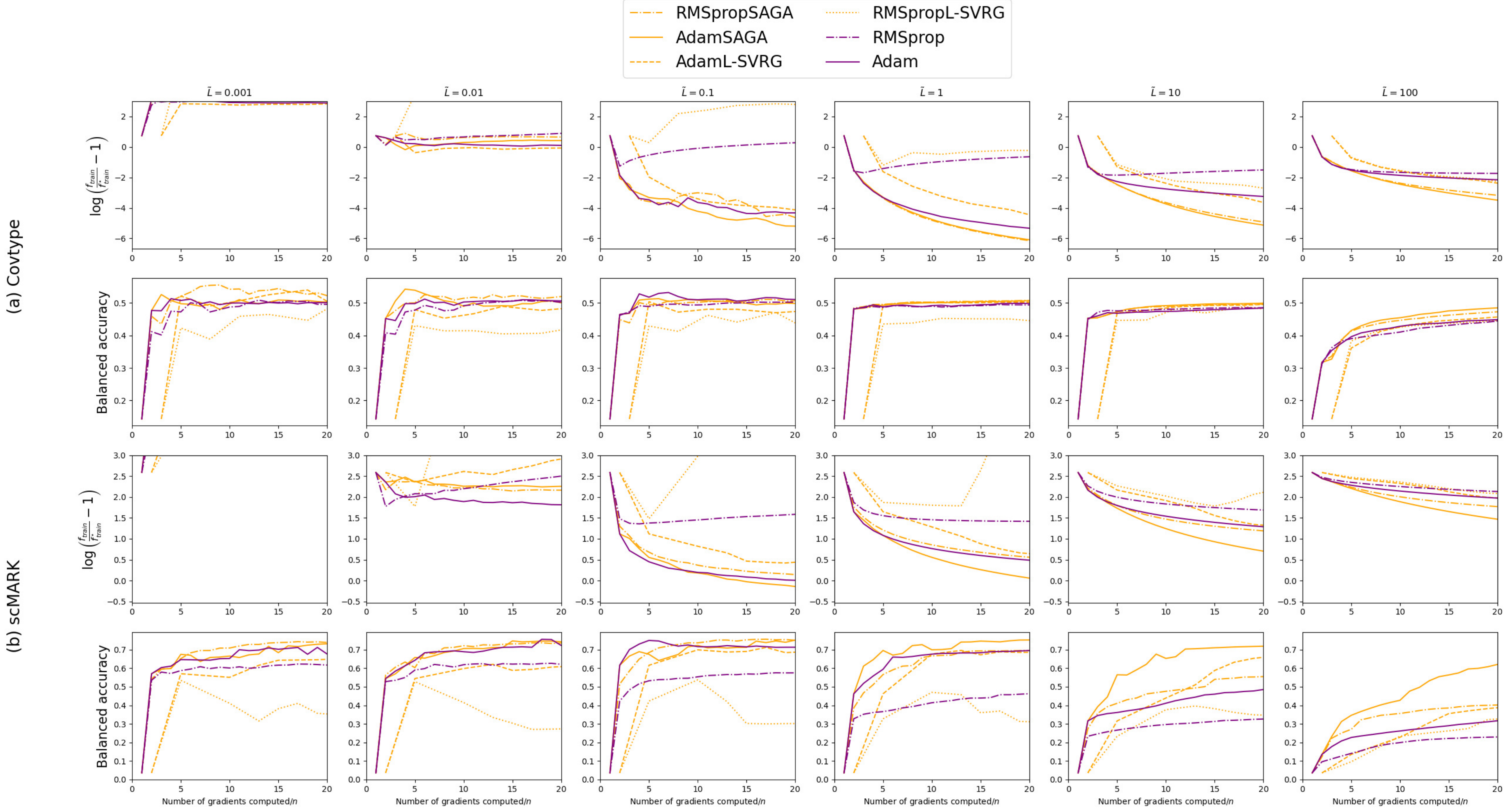}
\end{center}
\caption{Comparison of Adam and RMSprop with and without variance reduction.}
\label{fig:AdamRMSVR}
\end{figure*}

We define AdaLVR (Adaptive Loopless Variance Reduction) in \cref{alg:AdaLVR}, a family of algorithms which allows for four
different variants, by combining either AdaGrad-Norm or
AdaGrad-Diagonal with either SAGA or L-SVRG type gradient estimators.


\begin{restatable}{Theorem}{AdaLVR}\label{thm:AdaLVR}
Let $T \geq 1, \eta >0,  x^{(1)} \in \mR^d$ and $\left(\xt\right)_{1 \leq t
\leq T}$ be a sequence of iterates defined by \Cref{alg:AdaLVR} with $p=1/n$,
and where the AdaGrad variant (Norm or Diagonal) and the gradient estimator type
(SAGA or L-SVRG) are chosen at the beginning and remain constant.
Then, under \Cref{ass:convex-smooth,ass:minimum}, it holds that:
$$     \textstyle{\esp{f(\barre{x}_T)} - f\left(\xstar\right) \leq   \frac{
\alpha \left(\eta + \frac {D^2} {2\eta} \right) \sqrt{4Ln \Delta_1 } + 8L\alpha^2\left(\eta + \frac {D^2} {2\eta}\right)^2}{T}}$$
where $\barre{x}_T = \meanT \xt$, $\Delta_1= f\left(x^{(1)}\right) - f\left(\xstar\right)$ and $\alpha=1$ (resp.\ $\alpha=\sqrt{d}$) for the Norm variant (resp.\ the
Diagonal variant).
\end{restatable}
The above convergence guarantee corresponds to a gradient complexity
of $O(n+(L+\sqrt{nL})/T)$ which holds
\emph{adaptively} to the smoothness coefficient $L$.
\begin{Remark}
    In \Cref{ass:convex-smooth}, the convexity assumption of each $f_i$
        cannot be relaxed to the sole assumption of $f$ being convex as the
        analysis of
SAGA \& L-SVRG gradient estimators crucially rely on it.
The smoothness however, can be interestingly relaxed in the following fashion:
because the proposed method is adaptive to unknown smoothness, the
analysis can be slightly modified to establish an adaptivity to a
smoothness coefficient which only holds along the trajectory of
iterates. This "local" smoothness coefficient
may be much smaller (sometimes by several orders of magnitude) than
the global coefficient, or the latter may not even exist. However, we
chose to state our results with the global coefficient, for the sake
of simplicity and comparison purposes.
\end{Remark}
\section{Numerical experiments}
We conduct numerical experiments to assess the good empirical performance
of AdaLVR
over AdaGrad, SAGA, L-SVRG as well as state-of-the-art accelerated
variance-reduced algorithms.
We also consider as heuristics RMSprop and Adam combined with SAGA and
L-SVRG type gradient estimators.
We consider multivariate logistic regression with no regularization
on the Covtype
\parencite{misc_covertype_31}\bin{\footnote{http://archive.ics.uci.edu/dataset/31/covertype}}
and scMARK \parencite{scMark} datasets. The reproducibility of all the experiments can be effortlessly ensured using the provided code.zip file in the supplementary materials.
\paragraph{Experimental details and preprocessing}
The Covtype dataset has $n=581012$ samples ($30 \times 30$ meter cell) with 54 categorical features (cartographic variables). Each sample is labeled by its forest cover type, taking 7 different values.
scMARK is a benchmark for scRNA data (single-cell Ribonucleic acid) with
$n=19998$ samples (cells) and 14059 features (gene expression). Each sample is
labeled by its cell type, taking 28 different values. We only keep the 300 features
with largest variance.
A min-max scaler is applied as preprocessing to both datasets as well as a train-test split procedure ($n_{\text{train}} \simeq  0.8n$, $n_{\text{test}} \simeq 0.2n$).  We evaluate each algorithm
based on two metrics: \emph{the objective function} i.e.\ multivariate logistic
      regression loss computed on the train set and \emph{the balanced accuracy score} i.e.\ a weighted average of the accuracy score computed on the test set.

For the covtype (resp. scMARK) dataset, we use
a mini-batch size of $10$ (resp. $1$), and the dimension of the
parameter space is $d = 378 = 54\times 7$ (resp. $d = 8400=28\times300$). Experiments were run on a machine with 32 CPU cores and 64 GB of memory.

For most algorithms (including SAG, SAGA, (L-)SVRG, and accelerated variants), the hyperparameter $\eta$ must
typically be chosen proportionnally to the inverse of $L$ to yield the
best convergence guarantees, and larger values may lead to divergence.
However, such a choice often leads to poor practical performances and a larger
values of $\eta$ may result in faster convergence. Besides, the smoothness
coefficient $L$ may be unknown. Other algorithms such as AdaGrad
has guaranteed convergence for all values of its hyperparameter, but the
performance still depends on its value. The most common approach is to
tune $\eta$ via a ressource-heavy grid-search.
We compare the performance of the algorithms for several values of
their respective hyperparameters, each value corresponding to a guess
of the value of the smoothness parameter $L$. Although there is no a
priori relation between $L$ and the hyperparameter of adaptive
algorithms based on AdaGrad, RMSprop and Adam, their hyperparameter is
chosen on the same grid as the other algorithms, for ease of comparison.
To avoid any confusion with the true parameter $L$ resulting from the loss
function, we denote $\tilde L$ the coefficient given to the algorithm.

\paragraph{AdaLVR against SAGA, L-SVRG and AdaGrad}
Following customary practice, we skip the projection step in AdaGrad and its variants.
More attention will be put on the Diagonal versions of AdaGrad and AdaLVR since they outperform
the Norm versions in almost all cases.
We compare AdaLVR against SAGA, L-SVRG and AdaGrad in \Cref{fig:AdaptVRvsEach}.
We observe that adding the SAGA variance reduction gradient estimator to AdaGrad-Diagonal
gives a faster convergence of the objective function for $10$ out of $12$ plots (all except (b): $\tilde L =
0.001, 0.01$). The balanced accuracy is better or equivalent for AdaSAGA-Diagonal in every plot, and
significantly better in $7$ out of $12$ plots ((a): $\tilde L = 0.001, 1, 10,
100$, (b): $\tilde L = 1, 10, 100$).
In (a), AdaLVR-Diagonal gives similar balanced accuracy to SAGA but goes closer to the optimum and converges in more plots.
In (b), the AdaLVR-Diagonal is much more robust to the choice of the hyperparameter than SAGA or L-SVRG, giving
good balanced accuracy for $4$ plots out of $6$ ($\tilde L= 0.001, 0.01, 0.1, 1$), and
only $2$ for SAGA and L-SVRG ($\tilde L = 0.001, 0.01$). The perfomances of AdaLVR-Diagonal in terms of objective function are much better than SAGA or L-SVRG, going closer to the optimum in more plots.

Note that -Diagonal versions performs better than -Norm in the majority of cases, at the
small cost of an extra $O(d)$ memory storage. Moreover, SAGA computes $3$ times less gradients
per iteration than L-SVRG, although this is reflected on the plots as the
x-axis corresponds to the number of computed gradients; but stores $n$
gradients (one per sample) whereas L-SVRG stores only one gradient. If memory
is not an issue, AdaSAGA-Diagonal is recommended, otherwise AdaLSVRG-Diagonal is preferred.

An intriguing behavior happens for plots where the learning rate is high, as the
logistic loss and balanced accuracy seem to be positively correlated or not
correlated at all instead of negatively correlated. This can be explained as one can predict the wrong class and
have a relatively low or high loss. When the learning rate is high (i.e. $\tilde L$ close
to $0$), parameters tend to have a high loss when predicting the wrong class. This is
reflected if graphs are split between right and wrong predicted samples.

\paragraph{AdaLVR against accelerated and variance-reduced algorithms}
We compare AdaLVR against state-of-the-art accelerated algorithms (MiG,
Katyusha, L-Katyusha, VRADA, VARAG, DAVIS)  and adaptive accelerated algorithms
(AdaVRAG, AdaVRAE) in \Cref{fig:vsAccel}. We did not display AdaVRAE in
(a) as it did not give satisfying results.
We observe that non-adaptive accelerated algorithms are sensitive to the choice of the hyperparameter and converge only when its value is small
enough. By contrast, AdaLVR, AdaVRAE (only in (b)) and AdaVRAG are much more robust to the choice of the
hyperparameter, converging in more plots. Nevertheless, AdaLVR adapts
instantly and effectively to the smoothness, as the loss plummets on the first
epoch in the majority of plots while the decrease is much slower for AdaVRAG and AdaVRAE.

\paragraph{RMSprop and Adam with variance-reduced gradient estimator}
We here suggest the combination variance reduction with the RMSprop and
Adam algorithms, which are variants of AdaGrad-Diagonal, that
demonstrate much better practical performance.

AdaGrad accumulates the magnitute of the gradient estimators $\gt$ to adapt its stepsize, so that the
coefficients of $G_t$ (resp.\ $A_t^{-1}$) can only increase (resp.\
decrease) which may slow optimization down. RMSprop~\parencite{ruder2017overview}
remedies this by replacing the accumulation (meaning the regular sum) by a
discounted sum, thereby allowing the $G_t$ quantity to decrease:
\begin{equation} \label{eq:rmsprop_update}
    G_t = \gamma \diag\left(\gt \gtop\right) + \left(1 - \gamma \right)G_{t-1} ,\quad t \geq 1,
\end{equation}
with $0< \gamma <1$. Adam~\parencite{kingma2015adam} also adds momentum:
\begin{equation}\label{eq:adam_update}  \begin{array}{ll}
   & \mt  =  \beta_1 \mtm + \left(1 - \beta_1 \right) \gt \\
   & G_t  = \beta_2 G_{t-1} + \left(1 - \beta_2 \right) \diag\left(\gt \gtop\right),\\
   & \xtp = \xt - \eta G_t^{-\sfrac 1 2}\mt,
\end{array} \end{equation}
with $0 < \beta_1, \beta_2< 1 $. Both algorithms demonstrate excellent
performance and have given rise to many variants, even though their theoretical properties are less understood than AdaGrad.
We suggest replacing $\gt$ by SAGA or L-SVRG-type gradient estimator in \eqref{eq:rmsprop_update} and \eqref{eq:adam_update}. 
Numerical experiments are plotted in \Cref{fig:AdamRMSVR} (we set
$\gamma = 0.9$ for RMSprop and $\beta_1=0.9, \beta_2=0.999$ for Adam). Regarding
the objective function, the effectiveness of the method clearly
appears for RMSpropSAGA, AdamSAGA and AdamL-SVRG in $8$ plots out of
$12$ ($\tilde L= 0.1, 1,10,100$), while RMSprop struggles to converge when
combined with L-SVRG. The balanced accuracy is always greater or equal
when adding the SAGA-type gradient estimator, giving significantly
better results for $7$ plots out of $12$ ((a): $\tilde L=10,100$ and all
plots of (b) except $\tilde L=0.01$) and is sometimes multiplied by $2$
(AdamSAGA in (b): $\tilde L=100$, RMSpropSAGA in (b): $\tilde L=10,100$).

\section{Conclusion and perspectives}
In the context of the minimization of a sum of $L$-smooth convex
functions, we introduce AdaLVR, combining AdaGrad and SAGA or
L-SVRG-type gradient estimators. We demonstrate that AdaLVR takes the
best of both worlds, inheriting adaptiveness and low sensitivity to
the hyperparameter from AdaGrad, and fast convergence rates from
variance-reduced methods. We empirically show that Adam and RMSprop
benefit from variance reduction even more.

A natural follow-up of this work would be to combine AdaGrad with
acceleration techniques and loopless variance reduction to match the optimal
gradient complexity. Another improvement would be to
consider composite functions containing a possibly non-smooth
regularization term. Besides, the study of the nonconvex and
strongly convex cases are of interest. While not performed in
practice, we rely on the projection step to obtain convergence
guarantees, and an analysis of AdaLVR without projection would be
relevant. In addition, more general classes of variance reduced
algorithms---e.g. JacSketch~\parencite{gower2021stochastic}, Memorization
Algorithms~\parencite{hofmann2015variance}; may be analyzed along with
AdaGrad, allowing in particular to perform non-uniform sampling, the
use of mini-batches, and variants of SAGA with low memory requirements.
Finally, convergence guarantees for Adam and RMSprop combined with
variance reduction is an interesting line of research.

\onecolumn
\appendices
\addtocontents{toc}{\protect\setcounter{tocdepth}{0}}
\section{Proofs}
\AdaGradLemma*
\begin{proof}
  Let $1 \leq t \leq T$ and consider $z \triangleq \xt - \eta A_t ^{-1} \gt$ so that
  \begin{equation} \label{eq:update_rule_proof}
        z- \xstar = \xt - \xstar - \eta A_t^{-1} \gt.
\end{equation}
Multiplying both sides by $A_t$ and using $A_t A_t^{-1}\gt = \gt$ (\Cref{lem:pseudo-inverse}), we get

\begin{equation} \label{eq:A_t_update_rule_proof}
    A_t\left(z - \xstar\right) = A_t\left(\xt - \xstar\right) - \eta \gt.
\end{equation}
Taking the scalar product of \eqref{eq:update_rule_proof} and \eqref{eq:A_t_update_rule_proof}
\begin{align*}\ps{z- \xstar}{A_t\left(z - \xstar\right)} = & \ps{\xt -
\xstar}{A_t\left(\xt - \xstar\right)} - \eta \ps{\gt} {\xt -\xstar} \\
& -\eta \ps{A_t^
{-1}\gt} {A_t\left(\xt -\xstar\right)} + \eta^2 \ps{\gt}{A_t^{-1} \gt}.
\end{align*}
Using \Cref{lem:pseudo-inverse} once again gives

$$\ps{A_t^ {-1}\gt} {A_t\left(\xt -\xstar\right)} = \ps{A_tA_t^ {-1}\gt} {\xt -\xstar} = \ps{\gt} {\xt -\xstar} $$
so that
$$\ps{z - \xstar}{A_t\left(z - \xstar\right)} =\ps{\xt - \xstar}{A_t\left(\xt - \xstar\right)} - 2\eta \ps{\gt} {\xt -\xstar} + \eta^2 \ps{\gt}{A_t^{-1} \gt}. $$
Rearranging the terms gives
        \begin{align*}\ps{\gt}{\xt - \xstar} & = \frac \eta 2 \dln \gt\drn^2_{A_t ^{-1}} + \frac 1 {2\eta} \left(\dln \xt - \xstar \drn _{A_t} ^2 - \dln z - \xstar \drn _{A_t} ^2\right)\\
        & \leq \frac \eta 2 \dln \gt\drn^2_{A_t ^{-1}} + \frac 1 {2\eta} \left(\dln \xt - \xstar \drn _{A_t} ^2 - \dln x^{(t+1)} - \xstar \drn _{A_t} ^2\right).
        \end{align*}
    We used
      $ \dln \xtp - \xstar \drn_{A_t} \leq \dln z - \xstar \drn_{A_t}$
      since $\xtp = \PI{z}$. Summing the second term of the above
    right hand side gives
\begin{align*}
      \sumunT \left(   \dln \xt - \xstar \drn _{A_t} ^2 -\dln x^{(t+1)} -  \xstar \drn _{A_t} ^2\right)
     = &   \sumunT \dln \xt - \xstar \drn_{A_t} ^2 - \sum_{t = 2} ^{T+ 1} \dln \xt - \xstar \drn _{A_{t-1}} ^2  \\
   = & \sum_{t=1}^T \left(\dln \xt - \xstar \drn _{A_t} ^2 - \dln \xt - \xstar
   \drn _{A_{t-1}} ^2  \right) \\
     &+ \dln x^{(1)} -\xstar\drn _{A_0}^2 - \dln x^{(T+1)} - \xstar \drn _{A_T} ^2  \\
     \leq &\sum_{t=1}^T \dln \xt - \xstar \drn _{A_t - A_{t-1}} ^2  ,\\
\end{align*}
where $A_0 = 0$ so that $\dln x^{(1)} - \xstar
\drn ^2_{A_0} = 0$.  Now, using \Cref{lem:Lemma3} and the first inequality of \Cref{lem:Lemma4} gives
\begin{align*}
\sumunT \ps{\gt}{\xt - \xstar} & \leq \frac \eta 2 \sumunT  \dln \gt\drn^2_{A_t
^{-1}} + \frac 1 {2\eta} \left(\sum_{t=2}^T \dln \xt - \xstar \drn _{A_t - A_{t-1}}
^2 \right) \\
& \leq \left( \eta +  \frac{D^2}{2\eta} \right)  \Tr(A_T),
\end{align*}
Using the second inequality of \Cref{lem:Lemma4} ends the proof.
\end{proof}

\begin{Proposition}\label{prop:adatele}
    Let $T\geqslant 1$ and $\left(\xt\right)_{1 \leq t \leq T}$
        defined by \Cref{alg:AdaLVR}. Then,
        \[ \esp{\sumunT \esptnd{\gt}}  \leq \esp{8L\sumunT\left( f\left(\xt\right) -
    f\left(\xstar\right) \right)}  + 4L n\left(f\left(x^{(1)} \right) -
f\left(\xstar\right)\right). \]
\end{Proposition}

\begin{proof}
    Let 1 $\leq t \leq T$. We first deal with the case of SAGA-type gradient
    estimator. Using
    \cref{lem:normgtGen} with
    $y^{(t)}_i = \tildexti \left(1 \leq i \leq n\right)$ gives
\[\esptnd{\gt} \leq 2 \esptnd{\nabla f_{i(t)}\left(\xt \right) -  \nabla f_{i(t)}\left(\xstar\right) }  + 2 \esptnd{\nabla f_{i(t)} \left(\tildextmri\right)  -  \nabla
f_{i(t)}\left(\xstar\right)}.\]
Using \cref{coro:normDiffTat} for the first term and \cref{lem:normDiffPhiSAGA} for the second gives:
\begin{align*}\esptnd{\gt} \leq & 4L\left(f\left(\xt\right) -
f\left(\xstar\right)\right) \\
& + 4L\left(f\left(\xt\right) -
f\left(\xstar\right) +\sumunn D_{f_i}\left(\tildextmi,\xstar\right) -
\espt{\sumunn D_{f_i}\left(\tildexti ,\xstar\right)}\right) \\
 =&  8L\left(f\left(\xt\right) -
f\left(\xstar\right)
\right)+ 4L\left(\sumunn D_{f_i}\left(\tildextmi,\xstar\right) -
\espt{\sumunn D_{f_i}\left( \tildexti ,\xstar\right)}\right).
\end{align*}
Summing in $t$ and taking the expectation gives
\begin{align}
    \esp{\sumunT \esptnd{\gt}}  \leq &\esp{8L\sumunT\left(
    f\left(\xt\right) - f\left(\xstar\right) \right)} \nonumber \\
            & +\esp{4L \sumunT\left(\sumunn
D_{f_i}\left(\tildextmi ,\xstar\right) - \espt{\sumunn
D_{f_i}\left(\tildexti ,\xstar\right)}\right)} \nonumber\\
     = &\esp{8L\sumunT\left( f\left(\xt\right) - f\left(\xstar\right) \right)
    + 4L \sumunT\left(\sumunn D_{f_i}\left(\tildextmi,\xstar\right) -
\sumunn D_{f_i}\left(\tildexti,\xstar\right)\right)} \nonumber\\
     \leq &\esp{8L\sumunT\left( f\left(\xt\right) -
    f\left(\xstar\right) \right)}  + 4L\underbrace{\sumunn D_{f_i}\left(
    \tilde{x}_i^{(0)},\xstar\right)}_{= n\left(f\left(x^{(1)} \right) -
f\left(\xstar\right)\right)},
\end{align}
where we used the tower property at the third line, recognized a telescopic sum at
the last line and removed its last term since nonpositive, hence the result.

We turn to the case of L-SVRG type gradient estimator. Using \cref{lem:normgtGen} with
    $y^{(t)}_i = \wt \left(1 \leq i \leq n\right)$ gives
\[\esptnd{\gt} \leq 2 \esptnd{\nabla f_{i(t)}\left(\xt \right) -  \nabla f_{i(t)}\left(\xstar\right) }  + 2 \esptnd{\nabla f_{i(t)} \left(\wtm\right)  -  \nabla
f_{i(t)}\left(\xstar\right)}.\]
Using \cref{coro:normDiffTat} on the first term and \cref{lem:normDiffPhiLSVRG} on the second gives:
\begin{align*}\esptnd{\gt} \leq & 4L\left(f\left(\xt\right) -
f\left(\xstar\right)\right) \\
& + 4L\left(f\left(\xt\right) -
f\left(\xstar\right) +\sumunn D_{f_i}\left(\wtm,\xstar\right) -
\espt{\sumunn D_{f_i}\left(\wt ,\xstar\right)}\right) \\
 =&  8L\left(f\left(\xt\right) -
f\left(\xstar\right)
\right)+ 4L\left(\sumunn D_{f_i}\left( \wtm,\xstar\right) -
\espt{\sumunn D_{f_i}\left(\wt ,\xstar\right)}\right).
    \end{align*}
Summing in $t$ and taking the expectation gives
\begin{align}
    \esp{\sumunT \esptnd{\gt}} & \leq \esp{8L\sumunT\left(
    f\left(\xt\right) - f\left(\xstar\right) \right)  + 4L \sumunT\left(\sumunn
D_{f_i}\left(\wtm, \xstar\right) - \espt{\sumunn
D_{f_i}\left(\wt, \xstar\right)}\right)} \nonumber\\
    & = \esp{8L\sumunT\left( f\left(\xt\right) - f\left(\xstar\right) \right)
    + 4L \sumunT\left(\sumunn D_{f_i}\left(\wtm,\xstar\right) -
\sumunn D_{f_i}\left(\wt,\xstar\right)\right)} \nonumber\\
    &\label{eq:linkgtDeltat} \leq \esp{8L\sumunT\left( f\left(\xt\right) -
    f\left(\xstar\right) \right)}  + 4L\underbrace{\sumunn D_{f_i}\left(
    \tilde{x}^{(0)},\xstar\right)}_{= n\left(f\left(x^{(1)} \right) -
f\left(\xstar\right)\right)},
\end{align}
hence the result.\end{proof}

\AdaLVR*

\begin{proof}
    We set $\Delta_T \triangleq \sumunT \esp{ f\left(\xt\right) -
    f\left(\xstar\right) }$
  and $\gamma \triangleq \alpha \left( \eta + \frac{D^2} \eta \right)$ (where $\alpha=1$
  (resp.\ $\alpha=\sqrt{d}$) in
  the case of the Norm version (resp.\ the Diagonal version)). Using the convexity
  of $f$, Lemmas~\ref{lem:convenientBoundAdaGrad} and
  \ref{lem:SAGAunbiasedness}, and Proposition~\ref{prop:adatele},
\begin{align*}
\Delta_T&\leq \mathbb{E}\left[\sum_{t=1}^T \left< \nabla f(x^{(t)}), x^{(t)}-x^* \right>  \right] =\mathbb{E}\left[ \sum_{t=1}^T\mathbb{E}_t\left[ \left< g^{(t)}, x^{(t)}-x^* \right>  \right]  \right] \\
  &\leq \gamma\mathbb{E}\left[ \sqrt{\sum_{t=1}^T\dln g^{(t)} \drn^2} \right]\leq \gamma\sqrt{\mathbb{E}\left[ \sum_{t=1}^T\left\| g^{(t)} \right\|^2 \right] } \\
  &\leq \gamma\sqrt{\mathbb{E}\left[ 8L\sum_{t=1}^T\left( f(x^{(t)})-f(x^*) \right) \right]+4Ln\left( f(x^{(1)})-f(x^*) \right) }\\
  &=\gamma\sqrt{L\left( 8\Delta_T+4n(f(x^{(1)})-f(x^*)) \right)}.
\end{align*}
Using \Cref{lem:simple_inequality} gives
\[ \Delta_T \leq 8L \gamma^2+ \gamma\sqrt{4L
\left(f\left(x^{(1)}\right) - f\left(\xstar\right) \right)}. \]
By convexity, $\Delta_T\geq T \mathbb{E}\left[ f(\barre{x}_T)-f(x^*) \right]$
and the result follows.
\end{proof}

\section{Helper Lemmas }\label{sec:helperLemmas}
The Bregman divergence associated with a differentiable function
  $g:\mathbb{R}^d\to \mathbb{R}$ is defined as
  \[ D_g(y,x)=g(y)-g(x)-\left< \nabla g(x), y-x \right> ,\qquad x,y\in \mathbb{R}^d. \]
When $g$ is convex, the Bregman divergence is always nonnegative.
\begin{Lemma}[\parencite{Nesterov} Theorem 2.1.5] \label{lem:Lsmoothconseq}

If $g$ is an $L$-smooth convex  function

\begin{equation*}
    \frac 1 {2L} \left \| \nabla g(x) - \nabla g(y) \right\|^2  \leq D_{g}(y,x), \quad x,y \in \mR^d.
\end{equation*}
\end{Lemma}

\begin{Corollary} \label{coro:normGradfLeqf}
Under \Cref{ass:convex-smooth,ass:minimum},
$$ \dln \nabla f(x) \drn \leq 2L\left( f(x) - f\left(\xstar\right)\right), \quad x \in \mR^d. $$
\end{Corollary}
\begin{proof}
Let $x \in \mR^d$. Using \Cref{lem:Lsmoothconseq} gives
$$  \dln \nabla f(x) \drn =  \dln \nabla f(x) - \nabla f\left(\xstar\right) \drn \leq 2L D_f(x,\xstar) = 2L \left( f(x) - f\left(\xstar\right) \right). $$
\end{proof}

\begin{Corollary}\label{coro:normDiffTat}
    For \Cref{alg:AdaLVR} and under
    \Cref{ass:convex-smooth,ass:minimum}, for all $1 \leq t \leq T$, we have
\begin{equation*}
\esptnd{\nabla f_{i(t)}\left(\xt\right) -  \nabla f_{i(t)}\left(\xstar\right)} \leq 2L\left( f\left(\xt\right) - f\left(\xstar\right)\right).
\end{equation*}
\end{Corollary}
\begin{proof}
\begin{alignat*}{2}
    \esptnd{\nabla f_{i(t)}\left(\xt\right) -  \nabla f_{i(t)}\left(\xstar\right)} &
    = \frac {2L}{2L} \frac 1 n  \sumunn \dln \nabla f_{i}\left(\xt\right) -  \nabla f_{i}\left(\xstar\right) \drn^2  \\
    & \leq \frac {2L} n \sumunn D_{f_i}(\xt,\xstar) && \textrm{ (\Cref{lem:Lsmoothconseq})} \\
    & = \frac {2L} n \sumunn \left( f_i \left(\xt\right) - f_i\left(\xstar\right) -\ps{\nabla f_i\left(\xstar\right)}{\xt - \xstar}\right) \\
    & = 2L\left( f\left(\xt\right) - f\left(\xstar\right) - 0 \right),
\end{alignat*}
where in the last line we used the finite-sum structure  of $f$ and $\nabla f\left(\xstar\right) = 0$.
\end{proof}

\begin{Lemma}\label{lem:normDiffPhiSAGA}
    For \Cref{alg:AdaLVR} with \ref{eq:saga} and under
    \Cref{ass:convex-smooth,ass:minimum}, for all $1 \leq t \leq T$, we have
\begin{align*}
    \frac 1 {2L}\esptnd{\nabla f_{i(t)} \left(\tilde x ^{(t-1)}_{i(t)}\right)
-  \nabla f_{i(t)}\left(\xstar\right)} \leq  f\left(\xt\right) -
f\left(\xstar\right) +\sumunn D_{f_i}\left(\tilde x _i^{(t-1)}, \xstar\right) -
\espt{\sumunn D_{f_i}\left(\tilde x _i^{(t)},\xstar\right)}.
\end{align*}
\end{Lemma}

\begin{proof}
First, from \Cref{lem:Lsmoothconseq}
$$\frac 1 {2L} \espt{\dln \nabla f_{i(t)}\left(\tilde x ^{(t-1)}_{i(t)}\right) -
\nabla f_{i(t)}\left(\xstar\right) \drn^2}  \leq
\espt{D_{f_{i(t)}}\left(\tilde x_{i(t)}^{(t-1)},\xstar\right)} = \frac 1 n
\sumunn D_{f_i}\left(\tilde x _i^{(t-1)},\xstar\right). $$
By definition, for all $1\leq i\leq n$,  $\tilde x ^{(t)}_i|\xt = \left\{
    \begin{array}{ll}
 x ^{(t)}  &\mbox{with probability } 1/n \\
\tilde x ^{(t-1)}_i& \mbox{with probability } 1-1/n
\end{array} \right.$.
Therefore,
\begin{align*}
    \espt{\sumunn D_{f_i}\left(\tilde x_i^{(t)},\xstar \right)} =   \frac 1 n
    \sumunn  D_{f_i}(\xt,\xstar) + \frac {n-1}n \sumunn D_{f_i}\left(
    \tilde x _i^{(t-1)},\xstar\right).
\end{align*}
Reorganizing the terms and using $$ \frac 1 n \sumunn  D_{f_i}(\xt, \xstar) = \frac 1 n \sumunn \left( f_i\left(\xt\right) - f_i\left(\xstar\right) - \ps{\nabla f_{i}\left(\xstar\right)}{\xt - \xstar}\right) = f\left(\xt\right) - f\left(\xstar\right) - 0$$
gives
$$\frac 1 n \sumunn D_{f_i}\left(\tilde x_i^{(t-1)},\xstar\right) =
f\left(\xt\right) - f\left(\xstar\right) + \sumunn D_{f_i}\left(\tilde
x _i^{(t-1)},\xstar\right) - \espt{\sumunn D_{f_i}\left(\tilde x
_i^{(t)},\xstar\right)},$$
hence the result.
\end{proof}

\begin{Lemma}\label{lem:normDiffPhiLSVRG}
    For \Cref{alg:AdaLVR} with \ref{eq:LSVRG} and under
    \Cref{ass:convex-smooth,ass:minimum}, for all $1 \leq t \leq T$, we have
\begin{align*}
    \frac 1 {2L}\esptnd{\nabla f_{i(t)} \left(\wtm\right)
-  \nabla f_{i(t)}\left(\xstar\right)} \leq  f\left(\xt\right) -
f\left(\xstar\right) +\sumunn D_{f_i}\left(\wtm,\xstar\right) -
\espt{\sumunn D_{f_i}\left(\wt ,\xstar\right)}.
\end{align*}
\end{Lemma}

\begin{proof}
First, from \Cref{lem:Lsmoothconseq},
$$\frac 1 {2L} \espt{\dln \nabla f_{i(t)}\left(\wtm\right) -
\nabla f_{i(t)}\left(\xstar\right) \drn^2}  \leq
\espt{D_{f_{i(t)}}\left(\wtm,\xstar\right)} = \frac 1 n
\sumunn D_{f_i}\left(\wtm,\xstar\right). $$
By definintion, $\wt |\xt = \left\{
    \begin{array}{ll}
 x ^{(t)}  &\mbox{with probability } 1/n \\
\wtm & \mbox{with probability } 1-1/n
\end{array} \right.,$
therefore,
\begin{align*}
    \espt{\sumunn D_{f_i}\left( \wt,\xstar \right)} =   \frac 1 n
    \sumunn  D_{f_i}(\xt,\xstar) + \frac {n-1}n \sumunn D_{f_i}\left(
    \wtm,\xstar\right).
\end{align*}
Reorganizing the terms and using $$ \frac 1 n \sumunn  D_{f_i}(\xt,\xstar) = \frac 1 n \sumunn \left( f_i\left(\xt\right) - f_i\left(\xstar\right) - \ps{\nabla f_{i}\left(\xstar\right)}{\xt - \xstar}\right) = f\left(\xt\right) - f\left(\xstar\right) - 0$$
gives
$$\frac 1 n \sumunn D_{f_i}\left(\wtm,\xstar\right) =
f\left(\xt\right) - f\left(\xstar\right) + \sumunn D_{f_i}\left(
\wtm,\xstar\right) - \espt{\sumunn D_{f_i}\left(\wt,\xstar\right)}.$$
\end{proof}

\begin{Lemma}[\parencite{DefazioThese} Lemma 4.5]\label{lem:normgtGen}
    Under \Cref{ass:convex-smooth,ass:minimum}. Let $t\geq 1$ and  let $\gt$ be
    defined as
    \[ \gt = \nabla f_{i(t)}\left(\xt \right)  -  \nabla f_{i(t)}
\left(\anchor_{i(t)}\right)  + \frac 1 n \sumunn \nabla f_{i}(\anchor), \]
where $\anchor_i \in \mathbb R^d \left(1 \leq i \leq n \right)$ and $i(t)|\xt \sim \text{Unif}\{ 1, \dots, n\}$. We have,
\[\esptnd{\gt} \leq 2 \esptnd{\nabla f_{i(t)}\left(\xt \right) -  \nabla
    f_{i(t)}\left(\xstar\right)
    }  + 2 \esptnd{\nabla f_{i(t)} \left(\anchor_{i(t)}\right)  -  \nabla
f_{i(t)}\left(\xstar\right)}.\]
\end{Lemma}
\begin{proof}
We give the proof for completeness. Using the definition of $g^{(t)}$,
\begin{align*}
    \esptnd{g^{(t)}} & = \mE_t \bigg[ \bigg\| \nabla f_{i(t)}\left(\xt \right)  -  \nabla f_{i(t)}
    \left(\anchor_{i(t)}\right)  + \frac 1 n \sumunn \nabla f_{i}\left(\anchor_i\right)\bigg \|\bigg] \\
     & = \mE_t \bigg[\bigg \| \nabla f_{i(t)}\left(\xt \right) -  \nabla f_{i(t)}\left(\xstar\right) -
         \bigg( \nabla f_{i(t)} \left(\anchor_{i(t)} \right) -  \nabla f_{i(t)}\left(\xstar\right) -   \frac 1
             n \displaystyle \sum_{i=1}^n \nabla f_{i}\left(\anchor_i\right)
\bigg)\bigg \|^2 \bigg]  \\
    & \leq 2 \esptnd{\nabla f_{i(t)}\left(\xt \right) -  \nabla f_{i(t)}\left(\xstar\right)
        } \\
        & + 2 \mE_t \Bigg[\bigg \| \underbrace{\nabla f_{i(t)} \left(\anchor_{i(t)} \right) -
    \nabla f_{i(t)}\left(\xstar\right)}_{\triangleq X} -   \bigg[\underbrace{\frac 1 n
\sumunn \nabla f_{i}\left(\anchor_i\right) - \overbrace{\nabla f\left(\xstar\right)}^{= 0}}_{= \mathbb{E}_t\left[ X \right] }\bigg] \bigg \|^2
\Bigg] \\
    & \leq 2 \esptnd{\nabla f_{i(t)}\left(\xt \right) -  \nabla f_{i(t)}\left(\xstar\right)
    }  + 2 \esptnd{\nabla f_{i(t)} \left(\anchor_{i(t)} \right)  -  \nabla f_{i(t)}\left(\xstar\right)}
\end{align*}
where we used basic inequalities $\dln a + b \drn^2 \leq 2\dln a\drn^2 + 2\dln b\drn^2$ and $\mE_t\left[ \dln X - \mE_t\left[
X\right]\drn^2  \right] \leq \mE_t \left[ \dln X\drn^2  \right]$.
\end{proof}


\bin{
This result does not imply convergence of the objective function in general since the numerator can be unbounded if we don't make further  assumptions on $\gt$ and $f$. Assuming $L$-smoothness of $f$,  bounded gradients in expectation and $\espt{\gt} = \nabla f\left(\xt\right)$ where $\espt{\,\cdot\,} = \esp{\,\cdot\,|x^{(t-1)}}$, leads to an $\mathcal O\left(\frac 1 T + \frac 1 {\sqrt{T}}\right)$ convergence rate (\Cref{thm:AdaGradConv})}

\begin{Lemma}\label{lem:pseudo-inverse}
Let $\left( \gt\right)_{t \geq 1}$ be a sequence in $\mR^d$. For $t\geq 1$, and let $A_t$ be defined as either
\[ A_t = \sqrt{\sum_{s = 1}^t  \dln g^{(s)} \drn^2}\in \mR\qquad \text{or}\qquad A_t = \sqrt{\sum_{s = 1}^t \diag\left( g^{(s)} g^{(s) \top}
\right)}\in \mR^{d \times d},  \]
where the square-root is to be understood component-wise.
Then, $A_t A_t^{-1} \gt = \gt$, where $A_t^{-1}$ denotes the Moore--Penrose pseudo-inverse
of $A_t$.

\end{Lemma}
\begin{proof}
Let $t \geq 1$.
For the scalar case, if $\gt = 0 $ the identity holds. If $\gt \neq 0$, then $A_t>0$ so that $A_t A_t^{-1} =1$ and the identity holds.

For the diagonal case, let $1 \leq k \leq d$. We denote $B_{kk}$ the
$k^{\text{th}}$ diagonal term of a matrix $B \in \mathbb R^{d \times d}$. By definition of the
Moore--Penrose pseudo-inverse, since $A_t$ is diagonal,
\[ \left(A_t^{-1}\right)_{kk} = \begin{cases} \left(A_{t,kk}\right)^{-1} &\text{ if } A_{t,kk}>0\\ 0 &\text{if } A_{t,kk}=0
\end{cases}.  \]
Because $A_t$ is diagonal,
\[ \left(A_tA_t^ {-1} \gt\right)_{k} = \left(A_t^{-1}\right)_{kk}A_{t,kk} \gt_k. \]
If $\gt_k = 0$, then $\left(A_t^{-1}\right)_{k}A_{t,kk} \gt_k = 0.$ \\
If $\gt_k >0$, using the definition of $A_t$ gives $A_{t,kk}>0$ so that
 $$\left(A_t^{-1}\right)_{kk} A_{t,kk} = 1$$
 and the identity is satisfied.

\bin{$$ \left(A_t^ {-1}A_t \gt\right)_{i} = \left(A_t^ {-1}A_t\right)_{ii} \gt_i = \left(A_t^{-1}\right)_{ii}A_{t,ii} \gt_i =  \left(A_tA_t^{-1}\right)_{ii} \gt_i  = \left(A_tA_t^ {-1} \gt\right)_{i} = 0   $$
}
\end{proof}
The following lemma is adapted from \cite[Theorem 7]{Adagrad_article}.
\begin{Lemma}\label{lem:Lemma3}
Let $T\geq 1$ and $x\in X$. Then, for both AdaGrad-Norm and
AdaGrad-Diagonal as defined in \eqref{eq:adagrad}, it holds that
\[ \sum_{t=1}^T \dln \xt - x\drn^2_{A_t -A_{t-1}} \leq D^2 \Tr(A_T). \]
\end{Lemma}

\begin{proof}
Remember that $A_0 = 0$ by convention. For $1\leq t \leq T $, it holds that  $A_t - A_{t-1}\succeq 0$ by construction, so that

\begin{align*}
    \sumunT \dln\xt -x\drn^2_{A_t - A_{t-1}} & \leq  \sumunT \dln \xt - x\drn^2 \lambda_{\text{max}} \left(A_t- A_{t-1}\right) \\
    & \leq  \sumunT \dln \xt - x\drn^2 \Tr \left(A_t- A_{t-1}\right) \\
    &  \leq  \sumunT D^2 \Tr \left(A_t- A_{t-1}\right) \\
    & = D^2 \Tr(A_T).
\end{align*}
\end{proof}
The following lemma is adapted from \cite[Lemma 10]{Adagrad_article}.
\begin{Lemma}\label{lem:Lemma4}
Let $T\geq 1$. For both AdaGrad-Norm and AdaGrad-Diagonal as defined in \eqref{eq:adagrad}:
\[ \sumunT \dln g^{(t)} \drn^2_{A_t^{-1}} \leq 2 \Tr(A_T). \]
Moreover,
\[ \Tr(A_T) \leq \alpha \sqrt{\sumunT \dln \gt \drn^2} \]
where $\alpha=1$ (resp.\ $\alpha=\sqrt{d}$) for AdaGrad-Norm (resp.\ AdaGrad-Diagonal).
\end{Lemma}
\begin{proof}
    We proceed by induction. For $T = 1$, for the scalar version we get
    \[ g^{(1)\top} A_1^{-1}g^{(1)}   = A_1^{-1} \dln g^{(1)} \drn ^2   = A_1^{-1} A_1^2   =A_1 A_1^{-1} A_1 = A_1, \]
    where the last equality follows from the definition of the Moore--Penrose inverse.
For the diagonal version, we similarly have
\begin{align*}
    g^{(1)^\top} A_1^{-1}g^{(1)}  & = \Tr\left(A_1^{-1} g^{(1)}g^{(1)^\top}\right) = \Tr\left(A_1^{-1} \diag\left(g^{(1)}g^{(1)^\top}\right)\right) \\
    & = \Tr\left(A_1^{-1} A_1^2\right) = \Tr\left(A_1A_1^{-1} A_1\right) = \Tr \left( A_1\right),
\end{align*}
where we used the fact that if $A$ is diagonal and $B$ is arbitrary, then
$\Tr(AB) = \Tr(A\diag(B))$. This proves the result for $T = 1$. Let
$T'\geqslant 1$ and suppose it true for $T=T'-1$. For both variants (seeing $A_t$ as a
matrix of size 1 in the case of the scalar variant), we have
\begin{align*}
\sum_{t=1}^{T'}g^{(t)\top} A_t^{-1} \gt & = \sum_{t =1}^{T'-1}   g^{(t)\top} A_t^{-1} \gt + g^{(T')\top} A_{T'}^{-1} g^{(T')}\\
        & \leq 2 \Tr\left(A_{T'-1}\right) + g^{(T')\top}A_{T'}^{-1} g^{(T')}\\
    & = 2 \Tr \left(\left(A_{T'}^2 - g^{(T')}g^{(T')\top}\right)^{\sfrac 12}\right) + \Tr\left( A_{T'}^{-1} g^{(T')}g^{(T')\top}\right) \\
    &  \leq 2 \Tr(A_{T'})
\end{align*}
where we used the recurrence hypothesis at the second line  and at the last line we used  \cite[Lemma 8]{Adagrad_article}:
$$ X \succeq Y \succeq 0 \implies 2 \operatorname{Tr}\left(\left(X-Y\right)^{1 / 2}\right)+\operatorname{Tr}\left(X^{-1 / 2} Y\right) \leq 2 \operatorname{Tr}\left(X^{1 / 2}\right),$$
finishing the induction step. For the bound on the trace, for the norm version we have
$$\Tr(A_T) = G_T^{\sfrac 1 2} = \sqrt{\sumunT \dln \gt\drn^2  } $$
For the diagonal version we use Jensen's inequality to get
\begin{align*}
    \operatorname{Tr}(A_T) & = \operatorname{Tr}\left(G_T^{\sfrac 12}\right) = \sum_{j = 1}^d \sqrt{\lambda_j\left(G_T\right)} \\
    & =d \sum_{j = 1}^d \frac 1 d \sqrt{\lambda_j\left(G_T\right)}  \leq  d \sqrt{\sum_{j = 1}^d \frac 1 d \lambda_j\left(G_T\right)} \\
    & = \sqrt{d\Tr\left(G_T\right)}
\end{align*}
where $\lambda_j(G_T)$ denotes the $j^{\text{th}}$ eigenvalue of $G_T$. Now,
$$\Tr\left(G_T\right) = \Tr\left(\sumunT \diag \left( g^{(t) \top} \gt \right)\right) = \sumunT \dln \gt\drn ^2,$$
ending the proof.
\end{proof}

\begin{Lemma}\label{lem:SAGAunbiasedness}
    Let $1 \leq t \leq T $, and $\gt$ be defined by
 \Cref{alg:AdaLVR}. We have
\[ \mathbb{E}_t\left[ g^{(t)} \right]=\nabla f(x^{(t)}).  \]
\end{Lemma}
\begin{proof}
In the case of SAGA-type gradient estimators,
\begin{align*} \espt{\gt} & = \espt{\nabla f_{i(t)}\left(\xt\right)  -  \nabla f_{i(t)} \left(\tilde x^{(t-1)}_{i(t)}\right)  + \meann \nabla f_i\left(\tilde x^{(t-1)}_i\right)} \\
& = \meann \nabla f_{i}\left(\xt\right)  - \meann \nabla f_{i} \left(\tilde x^{(t-1)}_{i}\right)  + \meann \nabla f_i\left(\tilde x^{(t-1)}_i\right) \\
& = \nabla f\left(\xt\right).
\end{align*}

Similarly, in the case of L-SVRG gradient estimators,

\begin{align*} \espt{\gt} & = \espt{\nabla f_{i(t)}\left(\xt\right)  -
    \nabla f_{i(t)} \left(\wtm\right)  + \meann \nabla
f_i\left(\wtm\right)} \\
& = \meann \nabla f_{i}\left(\xt\right)  - \meann \nabla f_{i}
\left(\wtm\right)  + \meann \nabla f_i\left(\wtm\right) \\
& = \nabla f\left(\xt\right).
\end{align*}

\end{proof}



\begin{Lemma}\label{lem:simple_inequality}
If $x^{2} \leq a(x+b)$ for $a \geq 0$ and $b \geq 0$
$$
x \leq \frac{1}{2}\left(\sqrt{a^{2}+4 a b}+a\right) \leq a+\sqrt{a b}.
$$
\end{Lemma}
\begin{proof}
The starting point is the quadratic inequality $x^{2}-a x-a b \leq 0$.
Let $r_{1} \leq r_{2}$ be the roots of the quadratic equality, the
inequality holds if $x \in\left[r_{1}, r_{2}\right]$. The upper bound
is then given by using $\sqrt{a+b} \leq \sqrt{a}+\sqrt{b}$
$$
r_{2}=\frac{a+\sqrt{a^{2}+4 a b}}{2} \leq \frac{a+\sqrt{a^{2}}+\sqrt{4 a b}}{2}=a+\sqrt{a b}.
$$
\end{proof}

\section*{Acknowledgements}
We express our gratitude to Julien Chiquet for his meticulous review of the
manuscript and his leadership in guiding project meetings.

\ifCLASSOPTIONcaptionsoff
  \newpage
\fi

\printbibliography




\end{document}